\documentclass[a4paper]{amsart}
\usepackage[foot]{amsaddr}
\usepackage[english]{babel}
\usepackage[utf8]{inputenc}
\usepackage{times}
\usepackage[T1]{fontenc}
\usepackage{amsmath,amsthm,amssymb,enumerate}
\usepackage{xcolor}
\usepackage{tikz}
\usetikzlibrary{decorations.pathmorphing,positioning,patterns}
\usepackage[pdftikz,dessinempost]{optional}{\large }
\colorlet{couleurlien}{green!20!black}
\colorlet{couleurlink}{red!60!black}

\newcommand\sshift[1]{\ensuremath{\mathbf{#1}}}

\newcommand\ForbPat{\ensuremath{\mathcal{F}}}

\newcommand\engendre[1]{\ensuremath{\left<#1\right>}}

\newcommand\val[2]{\ensuremath{#1_{#2}}}

\newcommand\latin[1]{\emph{#1}}
\newcommand\ie{\latin{i.e.}, }

\newcommand\Z{\ensuremath{\mathbb{Z}}}
\newcommand\N{\ensuremath{\mathbb{N}}}

\newcommand\zent[2]{\ensuremath{Z_{#1}(#2)}}
\newcommand\commut[2]{\ensuremath{[#1,#2]}}

\newtheorem{theorem}{Theorem}[section]
\newtheorem{lemma}[theorem]{Lemma}

\theoremstyle{definition}

\theoremstyle{remark}

\newtheorem{claim}{Claim}[theorem]

\numberwithin{equation}{section}

\newtheorem{conjecture}[theorem]{\bf Conjecture}

\begin{document}

\title{The domino problem on groups of polynomial growth}

\author{Alexis Ballier}
\email{aballier@dim.uchile.cl}
\thanks{Alexis Ballier was supported by FONDECYT Postdoctorado Proyecto
3110088}
\author{Maya Stein}
\email{mstein@dim.uchile.cl \hskip-5pt{\tiny{(Corresponding author)}}}
\thanks{Maya Stein was supported by FONDECYT Regular grant 1140766 and N\'ucleo Milenio: Informaci\'on y Coordinaci\'on en Redes}
\address{Centro de Modelamiento Matem\'atico y Departamento de Ingenier\'ia Matem\'atica, Facultad de Ciencias F\'isicas y Matem\'aticas,
Universidad de Chile,
Beauchef 851, Torre Norte, Piso 7,
Santiago, RM, Chile.}

\subjclass[2010]{37B50, 03D35, 20F18}
\keywords{Domino problem, groups of polynomial growth, virtually free groups, decidability,  thick ends}

\begin{abstract}
        We conjecture that  a finitely generated group has a decidable domino problem if and only if it is virtually free. We show this is true for all virtually nilpotent finitely generated groups (or, equivalently, groups of polynomial growth), and for all finitely generated groups  whose center has a non-trivial, finitely
        generated  and torsion-free subgroup.\\
        Our proof uses a reduction of the undecidability of the domino problem on any such group $G$ to the undecidability of the domino problem on $\mathbb Z^2$, under the assumption that $G$ is not virtually free. This is achieved by first finding a thick end in $G$, and then relating the thick end to the existence of a certain structure, resembling a half-grid, by an extension of a result of Halin.
\end{abstract}

\maketitle

\section{Introduction}

In its original form, the domino problem consists of deciding whether the plane can be tiled with square unrotatable plates of equal sizes and coloured edges, coming from some previously fixed  finite set of plates, so that any two shared edges have the same colour. The problem was introduced by Wang~\cite{wangpatternrecoII} in 1961. Wang's student Berger showed undecidability for the domino problem on the (Euclidian) plane in 1964~\cite{bergerthesis, berger66}, using a reduction to the halting problem. In 1971, Robinson~\cite{robinson} simplified Berger's proof. 

The domino problem on the plane can be nicely expressed representing the plane by~$\Z^2$, and the tiles by symbols from some finite set. The colour restriction is translated by forbidding certain patterns: for instance symbol $a$ may not lie directly below symbol~$b$. Note that
forbidding larger patterns of fixed size does not change the
complexity of the problem.

Generalizations of the problems substitute $\Z^2$ with the Cayley graph of some finitely generated group or semi-group. (For a formal definition of the domino problem on groups see Section~\ref{sec:domino}.) We remark that decidability of the domino problem does not depend on the Cayley graph chosen, but on the (semi-)group itself (since the Cayley graph only depends on the set of generators we chose, and we can rewrite the generators for one graph in terms of the generators for the other graph).

Let us summarize the known results. Robinson~\cite{robinson2} conjectured that  the domino problem on the hyperbolic plane (which corresponds to the Cayley graph of a finitely generated semi-group) is undecidable. This was confirmed by Kari~\cite{kari}, and independently, by Margen\-stern~\cite{margen}. 
Aubrun and Kari~\cite{aubrunKari} showed that also on Baumslag-Solitar groups 
 the domino problem is undecidable. 
 
On the other hand, on finitely generated virtually free groups, the domino problem is decidable.
Indeed, the work of Muller and Schupp~\cite{MullerS83, MullerS85} from the 1980's, complemented with results of Kuske and Lohrey~\cite{KuskeL05} from 2005, gives that a finitely generated group is virtually free if and only if its Cayley graphs have a decidable monadic second order theory (MSO theory for short). Since the domino problem can be expressed in MSO theory (see for instance~\cite{JeandelT09}), we can conclude that all finitely generated virtually free groups have a decidable domino problem. 

We show that for finitely generated virtually nilpotent groups the converse is also true. That is, among all finitely generated virtually nilpotent groups, {\it only} the virtually free groups have a decidable domino problem. 

\begin{theorem}\label{decidomiimpliesfintwO}
For every virtually nilpotent and finitely generated group $G$, it holds that $G$ has a decidable domino problem if and only if $G$ is virtually free. 
\end{theorem}

We show in Section~\ref{sec:domino} how Theorem~\ref{decidomiimpliesfintwO} follows from the following result, Theorem~\ref{decidomiimpliesfintw}, which we believe to be of interest on its own. 

\begin{theorem}\label{decidomiimpliesfintw}
For any finitely generated group $G$ whose center has a non-trivial, finitely
        generated  and torsion-free subgroup, $G$ has a decidable domino problem if and only if $G$ is virtually free. 
\end{theorem}

For instance, the direct product of $\Z$ and any other group $G$ fall under the
hypothesis of Theorem~\ref{decidomiimpliesfintw}: $\Z\times\{1_G\}$ is
a non-trivial, finitely generated  and torsion-free subgroup of its center.
This includes, for example, the direct product of an infinite free Burnside
group or of the Grigorchuck group or of a Tarski monster
group with $\Z$, which fall outside of the hypothesis of
Theorem~\ref{decidomiimpliesfintwO} since they do not have polynomial growth.

       \smallskip

We now give an overview of the proof of 
Theorem~\ref{decidomiimpliesfintw}. We start with a finitely generated group~$G$ that is not virtually free, and whose center has a non-trivial, finitely generated  and torsion-free subgroup. Our aim is to show that $G$ does not have a decidable domino problem.

 As we shall see in detail in Section~\ref{sec:inftw}, a result of Woess~\cite{graphsgroupstrees} implies that $G$ has a thick end (see Section~\ref{sec:inftw} for a definition).
In Section~\ref{sec:thick}, we will show that any Cayley graph of a finitely-generated group~$G$ whose center has a non-trivial, finitely
generated and torsion-free subgroup with a thick end contains 
 a certain half-grid structure, resembling a subdivided $\mathbb N\times \mathbb Z$. This is an extension of a classical result by Halin from infinite graph theory.
         
Then, in Section~\ref{sec:domino}, we use the grid-like structure for a reduction of the undecidability of the domino problem. Namely, first we reduce the undecidability of the domino problem on a finitely-generated group $G$ that is not virtually free but has a non-trivial, finitely
generated and torsion-free center to the undecidability on $\Z ^2$, thus proving Theorem~\ref{decidomiimpliesfintw}. Then we use Theorem~\ref{decidomiimpliesfintw} for proving the analog for finitely generated  virtually nilpotent groups (Theorem~\ref{decidomiimpliesfintwO}). 
Some remarks on the (non-)applicability of our methods to other groups can be found in Section~\ref{conclu}.

\smallskip

We close the introduction with a conjecture due to the first author, suggesting that the equivalence of being virtually free and having a decidable domino problem holds for every finitely generated
group.

\begin{conjecture}
A finitely generated group $G$ has a decidable domino problem if and only if $G$ is virtually free.
\end{conjecture}

\smallskip

\section{Thick ends in finitely generated, not virtually free groups}\label{sec:inftw}

In this section we see that a finitely generated group $G$ which is not virtually free  has a thick end. We also see that such a thick end already appears in any finitely generated subgroup of finite index of $G$.

Whenever we consider a Cayley graph of a finitely generated group $G$, we tacitly assume that this Cayley graph is constructed using a finite set of generators of $G$. In particular, here we only consider  locally finite Cayley graphs of finitely generated groups.

We need to go through some notation. A {\em ray} in a graph is a one-way infinite path. An {\em end} is an equivalence class of rays, under the following equivalence relation: Two rays are equivalent if they are connected by infinitely many disjoint finite paths. Ends were first introduced by Freudenthal~\cite{Freudenthal31, Freudenthal42}. For example, the usual  Cayley graph of~$\Z$ has two ends, the usual Cayley graph of $\Z^2$ has one end, and any Cayley graph of a free group has infinitely many ends. It is not difficult to see that the number of ends is invariant under quasi-isometry.

A {\em thick end} in a graph is an end that contains an infinite set of disjoint rays. Thick ends were introduced by Halin~\cite{halingrid}. A group is said to have a thick end if one of its Cayley graphs has one. Of the three examples above, only $\Z^2$ has a thick end. 

Now, we relate the notion of thickness to the 
notion of the {\it diameter}
 of an end (see~\cite{graphsgroupstrees} for a definition).
Woess~\cite{graphsgroupstrees} showed that  every finitely generated group that is not virtually free has an end of infinite diameter. Together with Theorem~4.4 of~\cite{TW}, which states that in every connected, locally finite, vertex-transitive graph, every end of infinite diameter is thick, this  gives the following.

\begin{lemma}\label{inftwthick}
Every finitely generated group that is not virtually free has a thick end.
\end{lemma}

\smallskip

The next lemma is needed for the proof of Theorem~\ref{decidomiimpliesfintwO}.

\begin{lemma}
        \label{lemma:subgpthickend}
        Let $H$ and $G$ be finitely generated groups such that $H$ is a subgroup
        of finite index of $G$. Then $G$ has a thick end if and only if $H$ has
        a thick end.
\end{lemma}

For the proof of Lemma~\ref{lemma:subgpthickend}, we need another lemma, and for this we need to quickly recall the notion of quasi-isometry, an equivalence relation which is used for describing the large-scale similarity of two given metric spaces $X,X'$. A  function $\phi: X\to X'$ is called a {\it quasi-isometry} from $X$ to $X'$ if there are positive constants $c,C,\varepsilon$ such that $cd(x,y)-\varepsilon\leq d(\phi (x),\phi(y))\leq Cd(x,y)+\varepsilon$ for all $x,y\in X$, and if for each $x'\in X'$ there is some $x\in X$ such that $d(x', \phi(x))\leq C$. Spaces $X$, $X'$ are quasi-isometric if there is an quasi-isometry from $X$ to $X'$ (then there is also one from $X'$ to $X$).
 It is well known that Cayley graphs of a given  finitely generated group are unique up to quasi-isometry. So, in view of the following lemma, we see that in fact {\it  any} of the Cayley graphs of a group with a thick end has a thick end. 

\begin{lemma}[Woess~\cite{woess}, Lemma 21.4]\label{woesslemma}
Let $G$ and $G'$ be two quasi-isometric\footnote{Woess calls quasi-isometries `rough isometries', but it is the same notion.} locally finite graphs. Then the quasi-isometry extends to the ends of $G$ and $G'$, mapping thick ends to thick ends.
\end{lemma}

Now we are ready to prove Lemma~\ref{lemma:subgpthickend}.

\begin{proof}[Proof of Lemma~\ref{lemma:subgpthickend}]
        Corollary~IV.B.24 of~\cite{delaHarpe} (or Proposition 11.41 of~\cite{meier}) states that if $G$ is a finitely generated group and $H$ is a subgroup of finite index of $G$, then these groups are quasi-isometric. By Lemma~\ref{woesslemma}, we are done.
\end{proof}

\smallskip

\section{A structural property of groups with thick ends}\label{sec:thick}

 The main result of this section is Lemma~\ref{lemma:halingrp}. In this lemma, we show that every  finitely-generated group with a thick end, whose center has a non-trivial, finitely
generated and torsion-free subgroup  contains a structure roughly resembling the $\N\times \Z$ grid. This is an extension of a classical purely graph-theoretical result of  Halin~\cite{halingrid}. Halin's result states  that any infinite graph with a
thick end contains a subdivision of the $\N\times \Z$ grid (a subdivision of a graph $H$ is obtained from $H$ by replacing each edge with a path that has at least one edge; all these new paths have to be disjoint).

Define
\[
        \begin{array}{rcl}
                \commut{x}{y} & = & xyx^{-1}y^{-1},\\
                \zent{0}{G} & = & \left\{1_G\right\}, \\
                \zent{i+1}{G} & = & \left\{x\in{}G | \forall y\in{}G,
                    \commut{x}{y}\in\zent{i}{G}\right\}.
        \end{array}
\]

The subgroup $\zent{1}{G}$ is the {\em center} of $G$. A group $G$ is {\em nilpotent} if there
exists $n\in\N$ such that $\zent{n}{G}=G$.
An element $g\in{}G$ is called a {\em torsion element} if there exists $n\in\N$ such
that $g^n=_{G}1_G$. A group $G$ is called {\em torsion-free} if it does not contain any torsion
element.

\medskip

The following lemma is the heart of this section.

\begin{lemma}
        \label{lemma:halingrp}
        Let $G=\engendre{g_1,\ldots,g_n}$ be a finitely generated group with a
        thick end, such that $\zent{1}{G}$ has a finitely generated, non-trivial
        and torsion-free subgroup $Z$.
        Then there exists $a\in{}Z$ and a ray
        $w\in\left\{g_1,\ldots,g_n\right\}^{\N}$ such that no subword of $w$
        belongs to $\engendre{a}$. That is, for any $i<j\in\N$, $k\in\Z$,
         we have $w_i\ldots{}w_j\neq_{G}a^k$.
\end{lemma}

\begin{proof}
    Since  $Z$ is finitely
    generated,  the fundamental theorem of finitely
    generated abelian groups (see for instance~\cite{DummitFoote}) tells us that
        $Z$ is isomorphic to
    $\Z^n\oplus\Z_{k_1}\oplus\ldots\oplus\Z_{k_l}$, for some $n\in\N$ and
    $k_i\in\N, 1\leq{}i\leq{}l\in\N$.
    Since $Z$ is torsion-free, there is no such $\Z_{k_i}$ in the
    decomposition. Thus $Z$ is isomorphic to some
    $\Z^n$ with $n\geq{}0$. Moreover, since $Z$ is non-trivial, we know that $n\geq{}1$.
    
    If $n>1$, then let $a$ and $b$ be the elements of $Z$ that this
    isomorphism sends to $(1,0,0,\ldots,0)$ and $(0,1,0\ldots,0)$ in $\Z^n$. Let $w_i=b$ for all $i$; then the conclusion of the lemma follows easily.

    So from now on suppose $n=1$. Let $a$ be the element of $Z$ that is sent to
    $1\in\Z$ by the isomorphism between $Z$ and $\Z$; that is, $Z=\engendre{a}$.
    The discrete topology on $\left\{g_1,\ldots,g_n\right\}$ makes this space compact
    since it is finite, moreover, by the Tychonoff theorem,
    $\left\{g_1,\ldots,g_n\right\}^{\N}$ is also compact when embedded with the
    product topology.
    
    For any $m\in\N$, let $R_m$ be the set of all rays that have no non-empty subword equal  to
    any of $a^{-m},a^{-m+1},\ldots,a^m$ (where equality is taken in $G$).

\begin{claim}
\label{claim:mnn}
    For each $m\in\N$, the set $R_m$  is a non-empty compact subset of
    $\left\{g_1,\ldots,g_n\right\}^{\N}$.
\end{claim}

It is clear that the $R_m$'s from Claim~\ref{claim:mnn} are such that
$R_{m+1}\subseteq{}R_m$ and since they are all non empty, $\cap_{m\in\N}R_m$ is
also non empty by compactness.
Any element of $\cap_{m\in\N}R_m$ is a ray $w$ matching the conclusions.
This is enough to prove the lemma. So, it only remains to show Claim~\ref{claim:mnn}, which we will do in the remainder of the proof.

    By definition, $R_m$ clearly is a closed subset of the compact space 
    $\left\{g_1,\ldots,g_n\right\}^{\N}$, and is therefore
    compact too. Thus we only need to show that $R_m\neq \emptyset$, for $m\in\N$. Suppose otherwise. Then each of the rays
    $w\in\left\{g_1,\ldots,g_n\right\}^{\N}$ contains a subword equal to one of 
    $a^{-m},a^{-m+1},\ldots,a^m$. So, by compactness,
    there exists $K\in\N$ such that all the (now finite) paths
    $w\in\left\{g_1,\ldots,g_n\right\}^K$ contain a subword equal to one of 
    $a^{-m},a^{-m+1},\ldots,a^m$.
        
    Let $\{w^i\}_{i\in\N}$ be an infinite set of disjoint rays going to a thick
    end, with starting vertices~$v_i$.
    By disjointness, for all $i,j,k,k'\in\N$ with $i\neq j$ we have that
    \begin{equation}\label{disjointrays}
    v_iw^i_{1}\ldots{}w^i_k\neq{}v_jw^j_{1}\ldots{}w^j_{k'}.
    \end{equation}

    We claim that for any $i,k$, we can write $v_iw^i_{1}\ldots{}w^i_k=z_{i,k}a^{e_{i,k}}$ with
    $z_{i,k}$ being a word on the generators of $G$ and $e_{i,k}\in\Z$, such that 
    \begin{itemize}
            \item the length of $z_{i,k}$ as a word is less than $K$, and
            \item $|e_{i,k}-e_{i,k+1}|\leq{}m$.
    \end{itemize}
    Indeed, for a given $i$, we define $z_{i,k}$ and $e_{i,k}$ by induction on $k$. For $k=0$, note that
    since any path of length $K$ contains $a^j$, for some
    $j$ with 
    $-m\leq{}j\leq{}m$, and since $a^j$ commutes with all the elements of $G$, we have 
    $v_i=z_{i,0}a^{e_{i,0}}$, where $z_{i,0}$ is of length
    less than $K$ and $e_{i,0}\in\Z$. For the induction step, observe that
    as $a$ commutes with all the elements of~$G$, and by the induction hypothesis, we have
    $v_iw^i_1\ldots w^i_{k+1}=z_{i,k}w^i_{k+1}a^{e_{i,k}}$.
    If $z_{i,k}w^i_{k+1}$ has length less than $K$,  we define
    $z_{i,k+1}=z_{i,k}w^i_{k+1}$ and $e_{i,k+1}=e_{i,k}$.
    Otherwise, $z_{i,k}w^i_{k+1}$ is of length at least $K$ and can be written as
    $z_{i,k+1}a^j$, for $-m\leq{}j\leq{}m$, where $z_{i,k+1}$ is of length 
    less than $K$. We then define $e_{i,k+1}=e_{i,k}+j$.
    
    Let us now argue that $e_{i,k}\to_{k\to\infty}\pm\infty$. Indeed, otherwise  for an infinity of $k$'s, the elements
    $z_{i,k}a^{e_{i,k}}$ are at a bounded distance of
    $1_G$. Hence, the pigeonhole principle
    allows us to find $k\neq{}k'$ such that
    $$v_iw^i_1\ldots w^i_k=z_{i,k}a^{e_{i,k}}=z_{i,k'}a^{e_{i,k'}}=v_iw^i_1\ldots w^i_{k'}.$$ But this contradicts~\eqref{disjointrays}.

    We assume, without loss of generality, that for an infinity of
    $i$'s, $e_{i,k}\to_{k\to\infty}+\infty$.
    Since for every $i$ and $k$, $|e_{i,k+1}-e_{i,k}|\leq{}m$, by considering
    only the $e_{i,k}$'s that tend to $+\infty$ with $k$,
    for $I\in\N$, there exists $N_I\in\N$ such that for every $i\leq{}I$,
    there exists $k_i$ such that
    $e_{i,k_i}\in{}\left\{N_I-m,N_I-m+1,\ldots,N_I+m\right\}$.

    If $I$ is chosen sufficiently large, by the pigeonhole principle, we can find
    $i\neq{}i'$ such that $z_{i,k_i}=z_{i',k_{i'}}$ and
    $e_{i,k_i}=e_{i',k_{i'}}$ because both $z_{i,k_i}$'s and $e_{i,k_i}$'s
    belong to a finite set whose size does not depend on $I$.
    This means that $v_iw^i_1\ldots w^i_{k_i}=v_{i'}w^{i'}_1\ldots w^i_{k_{i'}}$ for
    some $i\neq{}i'$, again contradicting~\eqref{disjointrays}.
    
This completes the proof of Claim~\ref{claim:mnn} and thus the proof of the lemma.
\end{proof} \smallskip

\section{Decidability of the Domino Problem}\label{sec:domino}

In this section we reduce the undecidability of the domino problem in groups with the `half-grid' structure found in Section~\ref{sec:thick} to the undecidability of the domino problem on $\Z^2$. This will enable us to prove our main theorem.

\smallskip

We start by formally defining the domino problem. We need to introduce some notation first.
Let $\Sigma$ be a finite set, called the \emph{alphabet}, and let $G$ be a finitely generated
group, given together with a finite set of generators. Recall that for decidability of the domino problem it does not matter which finite set of generators was chosen.

The set $\Sigma^G$ is
called the \emph{fullshift}. An element of the fullshift is called a
\emph{configuration}.
For a finite set $D\subseteq{}G$, an element of $\Sigma^D$ is called a
\emph{pattern} over $D$.
For a configuration $x\in\Sigma^G$ and a pattern $P\in\Sigma^D$, we say that
\emph{$P$ appears in $x$} if there exists $i\in{}G$ such that for every
$j\in{}D$, $x_{ij}=P_j$.
A subset $\sshift{X}$ of $\Sigma^G$ is called a \emph{$G$-SFT} if it is
precisely the set of configurations of $\Sigma^G$ that do not contain any
pattern from a given \emph{finite} set of patterns $\ForbPat$.

In addition, we say that an SFT is \emph{one-step} if $\ForbPat$ contains only
elements of some $\Sigma^D$ with $D=\left\{1_G,g_i\right\}$ where $g_i$ is a
generator of $G$. One-step SFTs correspond to the intuition one has of a
tiling of $G$ and to the dominoes of Wang~\cite{wangpatternrecoII} on $\Z^2$.
It is well known that any SFT is conjugate to a one-step SFT, and the conjugacy
can be computed from its forbidden patterns and its alphabet.
For more information on SFT's and symbolic dynamics in general, see
\cite{lindmarcus}.

The \emph{domino problem} on $G$ consists of deciding if there exists a
configuration not containing any pattern of $\ForbPat$ when given $\ForbPat$;
equivalently, deciding whether the $G$-SFT defined by $\ForbPat$ is
non-empty~\cite{wangpatternrecoI,wangpatternrecoII}.

As a warm-up, we start with an easy lemma.

\begin{lemma}
        \label{lemma:subgroupundec}
        Let $H$ and $G$ be finitely generated groups such that $H$ is a subgroup
        of $G$.
        If the domino problem on $H$ is undecidable then so it is on $G$.
\end{lemma}

\begin{proof}
        Let $\sshift{X}_H$ be an $H$-SFT. The same $\sshift{X}_H$-forbidden
        patterns can be forbidden to obtain $\sshift{X}_G$, a $G$-SFT.
        It is clear that if $\sshift{X}_G$ is non-empty then so is
        $\sshift{X}_H$: For $x\in\sshift{X}_G$, $x_{|H}$ is an element of
        $\sshift{X}_H$.
        
        For the converse, i.e.~in order to see that  if $\sshift{X}_H\neq\emptyset$ then also
        $\sshift{X}_G\neq\emptyset$, proceed as follows. Consider the left cosets  $\left\{gH|
        g\in{}G\right\}$ of $H$ in $G$. 
        As these cosets form a partition of $G$, we can write $\left\{gH|
        g\in{}G\right\}=\left\{g_iH| i\in{}I\right\}$ such that
        $g_iH\neq{}g_jH$ if $i\neq{}j$. 
        
        There exist functions
        $f:G\to\left\{g_i|i\in{}I\right\}$ and $h:G\to{}H$ such that
        $g=f(g)h(g)$.
        For a configuration $c\in\sshift{X}_H$, we define $c'\in\Sigma^G$ by setting
        $\val{c'}{g}:=\val{c}{h(g)}$. Since all the forbidden patterns of
        $\sshift{X}_G$ are defined on $H$, we conclude that $c'$ does not contain any
        $\sshift{X}_G$-forbidden pattern (because $c$ does not contain any
        $\sshift{X}_H$-forbidden pattern). Thus $c'\in\sshift{X}_G$.

        We conclude that $\sshift{X}_H$ is non-empty if and only if
        $\sshift{X}_G$ is. Hence, if there exists an algorithm deciding the
        domino problem on $G$, such an algorithm can be used to decide it on
        $H$, completing the proof.
\end{proof}

The next lemma contains the reduction of the decidability of the domino problem on a group $G$ which contains the structure from Section~\ref{sec:thick}, to the decidability of the domino problem on $\Z^2$.

\begin{lemma}
        \label{lemma:undec}
        Let $G=\engendre{g_1,\ldots,g_n}$ be a finitely generated group such that $g_1$
        has infinite order, $g_1\in\zent{1}{G}$ and there exists a ray
        $w\in\left\{g_1,\ldots,g_n\right\}^{\N}$ such that no subword of $w$
        belongs to $\engendre{g_1}$, \ie for any $i<j\in\N$,
        $w_i\ldots{}w_j\not\in\engendre{g_1}$.
        Then the domino problem on $G$ is undecidable.
\end{lemma}

\begin{proof}
        We reduce to the $\Z^2$ case where the problem is already known to be
        undecidable~\cite{bergerthesis,robinson}. For this, we prove that there exists an
        algorithm which for any given  $\Z^2$-SFT $\sshift{X}$ computes a $G$-SFT $\sshift{X_G}$ such that $\sshift{X}$ is non empty if and only if
        $\sshift{X_G}$ is non empty.

        Let $\mathcal{A}(\sshift{X})$ denote the alphabet of $\sshift{X}$. We take the
        alphabet of  $\sshift{X_G}$ to be
        $\mathcal{A}(\sshift{X})\times\left\{2,\ldots,n\right\}$.
        Without loss of generality, we can assume that $\sshift{X}$ is a
        one-step SFT.
        The rules defining $\sshift{X_G}$ are as follows, for any
        $c\in{}(\mathcal{A}(\sshift{X})\times\left\{2,\ldots,n\right\})^G$ and
        any point $x\in{}G$:
        
        \begin{enumerate}[(I)]
                \item If $\val{c}{x}=(a,i)$ then $\val{c}{xg_1}=(b,i)$ and
                        $\val{c}{xg_1^{-1}}=(b',i)$  for some
                        $b,b'$.\\
                        In words: The second component is constant on the lines
                        defined by $\engendre{g_1}$.
                \item $\val{c}{x}=(a,i)$ and $\val{c}{xg_1}=(b,i)$ is allowed if
                        and only if $a$ is allowed left to $b$ in
                        $\sshift{X}$.\\
                        In words: The $\engendre{g_1}$ lines in $\sshift{X_G}$
                        represent the horizontal lines of $\sshift{X}$.
                \item $\val{c}{x}=(a,i)$ and $\val{c}{xg_i}=(b,j)$ is allowed if
                        and only if $a$ is allowed below $b$ in~$\sshift{X}$.\\
                        In words: The second component in the alphabet of
                        $\sshift{X_G}$ dictates the vertical direction and
                        following those paths represent the vertical direction
                        of $\sshift{X}$.
        \end{enumerate}
        
        Every other case is allowed. It is clear that one can compute the
        forbidden patterns defining $\sshift{X_G}$ when given those of $\sshift{X}$.
        It remains to prove that $\sshift{X_G}$ is non empty if and only if
        $\sshift{X}$ is.
        
        \medskip

        \textbf{If $\sshift{X_G}\neq\emptyset$, then $\sshift{X}\neq\emptyset$:}
        Let $c$ be a configuration of $\sshift{X_G}$.
        First, we define $c'$ only on $\mathcal{A}(\sshift{X})^{\Z\times\N}$. This will be done by induction as
        follows.
        
        For $i\in\Z$, choose $\val{c'}{(i,0)}$ so that $(\val{c'}{(i,0)}, x_0)=\val{c}{g_1^i}$, for some $x_0$. Then,  for
        $j\geq{}1$ and $i\in\Z$, let $\val{c'}{(i,j)}$ be such that
        $(\val{c'}{(i,j)}, y_j)=\val{c}{g_{x_{j-1}}g_1^i}$, where we define $g_{x_j}:=g_{x_{j-1}}g_{y_j}$ inductively.
        
        In order to see that $c'$ does not contain any
        $\sshift{X}$-forbidden pattern, observe no forbidden horizontal pattern
        may occur by rules (I) and (II), and by the definition of $c'$. For the
        vertical patterns, consider two points $a=c'_{(i,j)}$ and
        $b=c'_{(i,j+1)}$. Then by construction,
        $$\val{c}{g_{x_{j-1}}g_1^i}=(a,y_j)$$ and
        $$\val{c}{g_{x_{j-1}}g_1^ig_{y_j}}=\val{c}{g_{x_{j-1}}g_{y_j}g_1^i}=\val{c}{g_{x_{j}}g_1^i}=(b,y_j).$$
        Thus by rule (III), also no forbidden vertical pattern occurs.
        
          A standard compactness argument shows that we can extend $c$ to all of $\mathcal{A}(\sshift{X})^{\Z\times\Z}$. Thus
        $\sshift{X}$ is non empty.
        
        \medskip

        \textbf{If $\sshift{X}\neq\emptyset$ then $\sshift{X_G}\neq\emptyset$:}
        Let $c$ be a configuration of $\sshift{X}$. Recursively define a configuration $c'$ of $\sshift{X_G}$ as follows. At each step $s$,  the set of coordinates $L_s\subseteq{}G$ for which $c'$ is already
        defined will satisfy the following conditions:
        \begin{enumerate}[(a)]
        \item on $L_s$, configuration $c'$ does not contain any forbidden patterns,\label{a}
        \item if $x\in L_s$, then $xg_1^\ell\in L_s$ for all $\ell\in\Z$, and there are $k,n\in \mathbb Z$ and $j\in\{2,\ldots ,n\}$ such that $c'_x=(c_{(k,n)},j)$ and $c'_{xg_1^\ell}=(c_{(k+\ell,n)},j)$ for all $\ell\in \mathbb Z$,\label{b}
        \item if $x\in L_s$ with $c'_x=(z,j)$, then $xg_j\in L_s$.\label{c}
        \end{enumerate}

        We start by defining $c'$ for all points $w_1\ldots{}w_j$ of the ray $w$ from the assumption of the lemma, and all lines $w_1\ldots{}w_j\engendre{g_1}$. In other words, we take $$L_1:=\{w_1\ldots{}w_jg_1^\ell\ : \ j\in\mathbb N, \ell\in\mathbb Z\}.$$
        
        Define $$\val{c'}{w_1\ldots{}w_jg_1^\ell}:=(\val{c}{(\ell,j)},i)$$ for $\ell\in\Z$
        and $j\in\N$, where $i=i(j)$ is such that $g_i=w_{j+1}$.
        In this way $c'$ is well defined. Indeed,  if
        $w_1\ldots{}w_jg_1^\ell=w_1\ldots{}w_{j'}g_1^{\ell'}$ for some $\ell,\ell',j,j'$, then $j\neq j'$ since $g_1$ has infinite order. Therefore, we may assume that 
        $j'>j$, and hence
        $w_j\ldots{}w_{j'}=g_1^{\ell-\ell'}\in\engendre{g_1}$ contradicting our
        hypothesis on $w$.
        
        It is easy to check that conditions~\eqref{b} and~\eqref{c} are satified. For~\eqref{a}, first note that rules (I) and (II) clearly hold. For rule (III), suppose there is an $x=w_1\ldots{}w_j g_1^\ell$ with $\val{c'}{x}=(a,i)$ and $\val{c'}{xg_i}=(b,j)$ for some $a,b,i,j$ with $i\neq j$. Then by construction, we have that $g_i=w_{j+1}$. Further $a=\val{c}{(\ell,j)}$, and $b=\val{c}{(\ell,j+1)}$. Thus $a$ and $b$ are as necessary for rule (III).
        
        Now assume we are in step $s$, and wish to define $L_s$. Let $x\in{}G\setminus{}L_{s-1}$ be such that $xg_{m}\in{}L_{s-1}$ for some $1<m\leq{}n$. (We may assume such an $x$ exists, as otherwise we have defined $c'$ for all of $G$.) 
        
        By~\eqref{b}, also $xg_1^\ell\in L_{s-1}$, for all $\ell \in\mathbb Z$. Because of the second part of~\eqref{b}, there are   $k,n,j$ such that for all $\ell\in Z$ we have
        \[
        \val{c'}{xg_{m}g_1^\ell}=(\val{c}{(k,n)},j).
        \]
        
        We set $$L_s:=L_{s-1}\cup x\engendre{g_1}$$
   and define $$\val{c'}{xg_1^\ell}:=(\val{c}{(k+\ell,n-1)},m)$$ for all $\ell\in\mathbb Z$. Note that in this way $c'$ is well defined, as the $ xg_1^\ell$ are all distinct, and furthermore,  none of them is in $L_{s-1}$,  by~\eqref{b}, and since $x\notin L_{s-1}$.
   
        It is clear that~\eqref{b} holds in step $s$. Further,~\eqref{c} in step $s$ holds by the choice of $x$ and the definition of $c'$. So we only need to check~\eqref{a}, and again, rules (I) and (II) are easy.
        
        For rule (III), suppose there is an $x$ with $\val{c'}{x}=(a,i)$ and $\val{c'}{xg_i}=(b,j)$ for some $a,b,i,j$ with $i\neq j$. By~\eqref{a} for earlier steps, we may assume that one of $x$, $xg_i$ lies in $L_s\setminus L_{s-1}$. 
        We employ~\eqref{c} for step $s-1$ to see that if $xg_i\in L_s\setminus L_{s-1}$, then also $x\in L_s\setminus L_{s-1}$. So in all cases $x\in L_s\setminus L_{s-1}$.
By definition of $c'$, this means that $i=m$. Thus by construction, $a$, $b$ are as desired for rule (III).    
    
        \smallskip
        
        As in each step $s$ we add at least one element of $G\setminus L_{s-1}$ to $L_s$, after transfinitely many steps (namely, after at most $|G|$ steps) the union of all sets $L_s$ is $G$, and thus, using~\eqref{a}, we see that~$c'$ defines a configuration of $\sshift{X_G}$. Hence $\sshift{X_G}$ is
        non empty, as desired.
\end{proof}

We are now ready to prove Theorems~\ref{decidomiimpliesfintwO} and~\ref{decidomiimpliesfintw}.

\begin{proof}[Proof of Theorem~\ref{decidomiimpliesfintw}]
 Let $G$ be a  finitely generated group whose center has a non-trivial, finitely
        generated  and torsion-free subgroup. Assume $G$ is not virtually free.
                By Lemma~\ref{inftwthick} it follows that $G$ has a thick end.
         Hence, we may apply Lemma~\ref{lemma:halingrp} to $G$ and can deduce with the help of
        Lemma~\ref{lemma:undec} that the domino problem is undecidable on $G$.
\end{proof}

\begin{proof}[Proof of Theorem~\ref{decidomiimpliesfintwO}]
 Assume $G$ is finitely generated and virtually nilpotent, but not virtually free.
        Since any nilpotent group
        has a torsion free subgroup of finite index, we can choose a nilpotent and torsion-free subgroup $H$ of finite
        index of $G$.
        By Schreier's lemma, $H$ is finitely generated since it is a subgroup of finite
        index of a finitely generated group.
        
        By Lemma~\ref{inftwthick}, $G$ has a thick end and so, by Lemma~\ref{lemma:subgpthickend}, $H$ has a thick end.
        As~$H$ is nilpotent and finitely generated, also the subgroup $\zent{1}{H}$ is finitely generated (see for instance Lemma 1.2.2 of~\cite{stableGroups}). Moreover, as~$H$ is nilpotent, $\zent{1}{H}$ is non-trivial.
        Hence,
 we may apply Theorem~\ref{decidomiimpliesfintw} to see that the domino problem is undecidable on~$H$. By
        Lemma~\ref{lemma:subgroupundec}, the domino problem is undecidable on
        $G$ as well.
\end{proof}

\section{Conclusions}\label{conclu}

We conjecture the equivalence from Theorems~\ref{decidomiimpliesfintwO} and~\ref{decidomiimpliesfintw} holds for every finitely generated group but we
note that there are known cases that do not fall under the scope of our results: 
Most Baumslag-Solitar groups have a trivial center but none of them is virtually
free and all of them have an undecidable domino problem~\cite{aubrunKari}.
Lemma~\ref{lemma:halingrp} is really focused on exhibiting the $\Z^2$-like
structure of certain groups, which allows us to carry out a reduction to the domino problem
on $\Z^2$; on the other hand, Aubrun and Kari~\cite{aubrunKari} give a new proof
of the undecidability of the domino problem adapted to the structure of the
groups they are studying. This latter approach is probably needed for
Baumslag-Solitar groups since it seems difficult to find a $\Z^2$-like structure
withing these groups. However, it may be possible to combine their 
methods with ours in order to find a larger class of groups with undecidable
domino problem.

\section*{Acknowledgment}
We would like to thank the anonymous referees for their useful remarks.

\bibliographystyle{plain}
\bibliography{../../../biblio/biblio}

\end{document}